\documentclass[11pt,reqno]{amsart}

\usepackage[all]{xy}

\numberwithin{equation}{section}

\usepackage{amsfonts,amsmath,amsthm}
\usepackage{amssymb,epsfig}
\usepackage{ascmac}  
\usepackage{cases, verbatim}
\usepackage[usenames]{color}
\usepackage{esint}
\usepackage{setspace}

\usepackage{enumerate} 

\usepackage{mathtools}
\usepackage{color} 
\definecolor{vert}{rgb}{0,0.6,0}

\usepackage[sort, numbers]{natbib}
\theoremstyle{plain}
\newtheorem{thm}{Theorem}

\newtheorem{defn}{Definition}

\newtheorem{lem}[thm]{Lemma}

\theoremstyle{remark}

\newcommand{\R}{\mathbb{R}}

\newcommand{\e}{\epsilon}
\newcommand{\p}{\partial}

\begin{document}
\title[Strong Comparison Principle]
{Strong Comparison Principle for $p-$harmonic functions in Carnot-Caratheodory spaces}

\author[L. Capogna]{Luca Capogna}
\address{Luca Capogna, Department of Mathematical Sciences, Worcester Polytechnic Institute, Worcester, MA 01609, USA, {\tt lcapogna@wpi.edu}}

\author[X. Zhou]{Xiaodan Zhou}
\address{Xiaodan Zhou, Department of Mathematical Sciences, Worcester Polytechnic Institute, Worcester, MA 01609, USA, {\tt xzhou3@wpi.edu}}

\keywords{ subelliptic PDE,  strong comparison principle, $p$-harmonic functions, sub-Riemannian geometry
L.C. was partially funded by NSF awards  DMS 1449143  and DMS 1503683.
X. Z. was partially funded by   an   AMS-Simons Travel Grant}
\date{\today}

\maketitle
\begin{abstract}
We  extend Bony's   propagation of support argument \cite{Bony} to  $C^1$  solutions of the non-homogeneous sub-elliptic $p-$Laplacian associated to a system of smooth vector fields satisfying H\"ormander's finite rank condition. As a consequence we prove a strong maximum principle and strong comparison principle that generalize  results of Tolksdorf \cite{Tolksdorf}.
\end{abstract}
\section{Introduction}

Let $\Omega\subset \mathbb{R}^n$ be an open and connected set, and consider a family of smooth vector fields $X_1, \cdots, X_m$ in $\mathbb{R}^n$ satisfying H\"ormander's finite rank condition \cite{Hormander},
\begin{equation}\label{hormander}
{\rm rank\  Lie}[X_1, \cdots, X_m](x)=n,
\end{equation}
for all $x\in \Omega$.  We set $X u=(X_1 u, \cdots, X_m u)$ for any function $u:\Omega\to \R$ for which the expression is meaningful.

 In this paper we will prove a strong comparison principle for solutions of the  class  of quasilinear, degenerate elliptic equations 
\begin{equation}\label{PDE}
L_p u=\sum_{j=1}^m X_j^* (A_j(X u))=f(x,u),
\end{equation}
satisfying the structure conditions \eqref{structure condition}, and which includes the $p-$Laplacian, in the range $p>1$,
associated to $X_1,...,X_m$ and to  the Lebesgue measure $dx$ in $\R^n$. 
Note that in \eqref{PDE} we have let  $X_j^*=-X_j+d_j(x)$ denote the $L^2$ adjoint of  the operator $X_j$ with respect to the Lebesgue measure. Here $d_j$ is a smooth function obtained as the trace of $X_j$. We explicitly note that all the results in this paper continue to hold if one   substitutes the Lebsgue measure $dx$ with any other measure $d\mu=\lambda(x) dx$ with $\lambda\in C^1$ density function. In particular the results apply in any subRiemannian manifold, for  solutions of the subelliptic $p-$Laplacian associated to a smooth volume form.    
%
%
%

In addition to the structure conditions \eqref{structure condition}, our strong comparison principle holds under the following hypothesis: 
\begin{equation}\label{f}
\begin{aligned}  & (i) \ \ \partial_u f\le 0 \text{ in  }\Omega,\\ 
&   (ii) \ \ | f(x,u_2+\e)-f(x,u_2)|\le  L\epsilon, \text{ for any } \epsilon \in [0,\epsilon_0], x\in \Omega\\
 \end{aligned}
\end{equation}
for some positive constants $L, \epsilon_0$.
Our main result is the following

\begin{thm}[Strong Comparison Principle]\label{main}
Let $\Omega\subset \R^n$ be a connected open set and consider two weak solutions $u_1\in C^1(\bar \Omega)$, and $u_2\in C^2(\bar \Omega)$ of \eqref{PDE} in $\Omega$, with $ |X u_2|\ge \delta \text{ in }\Omega$. We assume that the structure conditions \eqref{structure condition}, and the hypothesis \eqref{f} are satisfied.
If $$u_1\ge u_2 \text{  in }\Omega, $$ then either $u_1=u_2$ or $$u_1> u_2 \text{ in }\Omega.$$
\end{thm}

As it will be evident from the proof, the regularity assumptions and the lower bound on $|Xu_2|$ are required only in a neighborhood of the contact set. The lower bound is not required in the non-degenerate case $\kappa>0$.

Bony's method can also be used to establish a non-homogenous strong maximum principle. We suppose that $f$ satisfy the following conditions: for all $x\in \Omega$ and $u\in\R$,
\begin{equation}\label{mf}
\begin{aligned}  & (i) \ \ \partial_u f\le 0 ,\\ 
&   (ii) \ \ | f(x,u)|\le \bar C (\kappa+|u|)^{p-2}|u|  \\
 \end{aligned}
\end{equation}
for some positive constant $\bar C$ and $\kappa$ as in the structure conditions \eqref{structure condition}.
\begin{thm}[Strong Maximum Principle]\label{max}
Let $\Omega\subset \R^n$ be a connected open set and consider a weak solution $u\in C^1(\bar \Omega)$ of \eqref{PDE} in $\Omega$. We assume that the structure conditions \eqref{structure condition} and the hypothesis \eqref{mf} hold.
If $$u\ge 0 \text{  in }\Omega, $$ then either $u=0$ or $$u> 0 \text{ in }\Omega.$$
\end{thm}

The proof of these results is at the end of Section \ref{3}.
 Theorem \ref{main} and Theorem \ref{max}  extend to the subelliptic setting the strong maximum and comparison principles proved by Tolksdorff  in \cite[Propositions 3.2.2 and  3.3.2]{Tolksdorf}.  

In the subelliptic setting Theorem \ref{main} seems to be new even in the homogeneous case $f=0$. In terms of previous literature on this subject: we recall that  the case $p=2$ was established through geometric methods by Bony in his landmark paper \cite{Bony}. A proof of the strong maximum principle for the subelliptic $p-$Laplacian in $H-$type groups can be found in \cite{YN}. We note however that at the conclusion of that proof the authors claim that one can always fit a gauge ball tangentially at every point of the set where the solution attains the maximum. This statement is not proved in \cite{YN}, and since gauge balls have zero curvature at the poles, we do not seem how it can be proved. 

A strong comparison and maximum principle for smooth solutions of the subelliptic $p$-Laplacian and of the horizontal mean curvature operator has been recently proved by Cheng, Chiu, Hwang and Yang  in their preprint \cite{CCHY}. Their proof is based on a linearization approach which is different from our arguments,  however it also ultimately relies  on Bony's argument, and holds in every subRiemannian manifold.  In comparison to the present paper, on the one hand  our results hold for solutions which do not have to be smooth necessarily\footnote{We recall that in general $p-$harmonic functions do not enjoy more regularity than the H\"older continuity of their gradient.}, but for the comparison principle we require one of the two solutions to have non-vanishing horizontal gradient.  On the other hand while we only deal with the $p-$Laplacian,  in \cite{CCHY} the authors also establish far reaching results for the mean curvature operator, including some special cases where $|Xv_2|$ is allowed to vanish in a controlled fashion and still have a comparison principle.

 The technical core of  the  proofs in the present paper is in Lemma \ref{hopf} and consists in an adaptation of  Bony's argument to our nonlinear setting. Note that, as in the Euclidean setting, one cannot relax the conditions on $u_1$, $u_2$ and $f$ unless more hypothesis are added.

In closing we note that  both in the elliptic and in the subelliptic case, a  corresponding strong maximum principle for the homogenous problem, $f=0$, can be established immediately from the Harnack inequality (see for instance \cite{BHT}, \cite{HH}, \cite{CDG1}), as well as with small modifications of the argument presented here. However, while in the linear setting one can deduce the strong comparison principle from the strong maximum principle, this is no longer the case in the nonlinear setting, where a new approach is needed.

\section{Bony's propagation of support technique}
Tolksdorf's argument in \cite[3.3.2]{Tolksdorf} breaks down in the subelliptic setting, due to the fact that the horizontal gradient of the barrier functions typically used in this proof may vanish.
The same problem occurs also in the linear setting, for $p=2$. To deal with this issue we follow the outline of the proof of the strong maximum principle for subLaplacians, from Bony's paper \cite{Bony}, and adapt it to our non-linear and non-homogeneous setting.

We begin by recalling from \cite[Definition 2.1]{Bony}, the definition of a nonzero vector $\mathbf{v}$ orthogonal to a set $F\subset \R^n$ at a point $y\in \partial F$. 

 \begin{defn} 
Let $F$ be a relatively closed subset of $\Omega$. We say that a vector $\mathbf{v}\in \mathbb{R}^n\setminus\{0\}$ is (exterior) normal to $F$ at a point $y\in \Omega\cap\partial F$ if
\[
\overline{B(y+\bf{v},|\bf{v}|)}\subset (\Omega\setminus F)\cup\{y\}.
\]
If this inclusion holds, we write $\mathbf{v}\perp F$ at $y$. Set
\[
F^{*}=\{y\in \Omega\cap \partial F\colon\mbox{there exists}\  \mathbf{v}\  \mbox{such that}\  \mathbf{v} \perp F\  at \ y\}.
\]
\end{defn}
Note when $\Omega$ is connected and $\emptyset \neq F\neq \Omega$, we have $F^{*}\neq \emptyset$.

We list in the following some of the results  and definitions from \cite{Bony} that play a role in our proof.

\begin{defn} Let $X$ be vector field in $\Omega$ and $F\subset \Omega$ be a closed set. We say that $X$ is tangent to $F$ if, for all $x_0\in F^*$ and all vectors $v$ normal to $F$ at $x_0$ one has that their Euclidean product vanishes, i.e.  $\langle X(x_0), v\rangle=0$.
\end{defn}

The following results are  from \cite[Theoreme 2.1]{Bony}, and \cite[Theoreme 2.2]{Bony}:
\begin{thm}\label{B1} Let $\Omega\subset \R^n$ be an open set and $F\subset \Omega$ a closed subset. Let $X$ be a Lipschitz vector field in $\Omega$.  If $X$ is tangent to $F$ then all its integral curves that intersect $F$  are entirely contained in $F$.
\end{thm}
Note that the converse of this result is also true, and follows from a direct computation.

\begin{thm}\label{B2} Let $\Omega\subset \R^n$ be an open set and $F\subset \Omega$ a closed subset. Let $X_1,...,X_m$ be  smooth vector fields in $\Omega$.  If $X_1,...,X_m$ are tangent to $F$ then so is the Lie algebra they generate.
\end{thm}

As a corollary, if $X_1,...,X_m$ satisfy H\"ormander finite rank condition \eqref{hormander} and are all tangent to $F$ then every curve that touches $F$ is entirely contained in $F$, so that either $F$ is the empty set or $F=\Omega$.

%
%
%
%
%
%
%
\section{A Hopf-type comparison principle and proof of Theorem \ref{main}}\label{3}
First we state precisely  the structure conditions  imposed on the left hand side of  \eqref{PDE}.
  The functions $A_j$ satisfy the following ellipticity and growth condition: For $p>1$, for a.e.  $\xi\in \R^m$ and for every $\eta\in \R^m$, 
\begin{equation}\label{structure condition}
\begin{aligned}
&\sum_{i,j=1}^m \frac{\partial A_j}{\partial\xi_i}(\xi)\eta_i\eta_j \ge\beta(\kappa+ |\xi|)^{p-2} |\eta|^2\\
&\sum_{i,j=1}^m |\frac{\partial A_j}{\partial\xi_i}(\xi)|\le \gamma(\kappa+ |\xi|)^{p-2}
\end{aligned}
\end{equation}
for some positive constants $\beta, \gamma, \kappa$.  

One can easily deduce that there exists positive constant $\lambda,C$ such that for all $\xi \in \R^m$,
\begin{equation}\label{weakinequality}
\langle A_j(\xi)-A_j(\xi'), \xi-\xi'\rangle \ge \lambda \begin{cases}
       (1+|\xi|+|\xi'|)^{p-2}|\xi-\xi'|^2  &\quad\text{if}\quad p\le 2\\
     
       |\xi-\xi'|^p &\quad\text{if}\quad p\ge 2, \\ 
     \end{cases}
\end{equation}
and 
$$|A_j(\xi)|\le C (\kappa+|\xi|)^{p-2}|\xi|.$$
The subelliptic  $p-$Laplacian 
\[
L_p u=\sum_{j=1}^m X_j^* (|X u|^{p-2}X_j u),
\]
corresponds to the choice $A_j(\xi)=|\xi|^{p-2}\xi_j$ for $j=1, \cdots, m$.

We will need the following immediate consequence of the monotonicity inequality \eqref{weakinequality}.
\begin{lem}[Weak Comparison Principle]\label{weak}
Let $\Omega\subset \mathbb{R}^n$ be an open and connected set and $v_1, v_2\in C^1(\Omega)$ satisfy in a weak sense
\begin{equation}
\begin{cases}  L_pv_2  \le f(x,v_2) & \mbox{in } \Omega\\  L_p v_1\ge f(x,v_1) & \mbox{in } \Omega, \end{cases}
\end{equation}
with $A_j$ satisfying the structure conditions \eqref{structure condition} and $\p_u f(x,u) \le 0$.
If $v_2\le v_1$ in $\p \Omega,$ then $v_2\le v_1$ in $\Omega$.
\end{lem}

\begin{proof}
Given an arbitrary $\epsilon>0$, we define $E_\epsilon=\{x\in \Omega| v_2(x)>v_1(x)+\epsilon\}$. Assume that $E_\epsilon\neq \emptyset$, then $\overline{E_\epsilon}\subset \Omega$. 
For all $\varphi \in C_c^1(\Omega)$, we have
\[
\int_\Omega \langle A_j(X v_2), X\varphi\rangle\le \int_\Omega f(x,v_2) \varphi ,
\]
\[
\int_\Omega \langle A_j(X v_1), X\varphi\rangle\ge \int_\Omega f(x,v_1) \varphi.
\]
Subtracting the above two inequalities and setting $\varphi(x)=\max\{v_2(x)-v_1(x)-\epsilon, 0\}$ then as a consequence of $(ii)$ in \eqref{f}, one has
\[
\int_{E_\epsilon} \langle A_j(X v_2)-A_j(X v_1), X(v_2-v_1)\rangle\le \int_{\{v_2>v_1+\e\}}  (f(x,v_2)-f(x,v_1)) (v_2-v_1-\e)   \le 0.
\] 
By \eqref{weakinequality}, this inequality holds if and only if $X(v_2-v_1)=0$. Thus, $v_2=v_1+C$ in $E_\epsilon$. The fact that $v_2=v_1+\epsilon$ on $\partial E_{\epsilon}$ implies that $C=\epsilon$. It follows that $v_2\le v_1+\epsilon$ in $\Omega$. Let $\epsilon\to 0$, we get $v_2\le v_1$ in $\Omega$. 
\end{proof}

Next, we prove an analogue of the classical Hopf comparison principle: Given a subsolution $v_2$ and a supersolution $v_1$ such that $v_2 \le v_1$, then every vector field $X_1,...,X_m$ must be tangent to the contact set $F=\{v_2=v_1\}$.

\begin{lem}{(A Hopf-type Comparison Principle)}\label{hopf} Let $\Omega\subset \mathbb{R}^n$ be an open and connected set  and $v_1\in C^1(\Omega)$, $v_2\in C^2(\Omega)$ with $ |X v_2|\ge \delta \text{ in }\Omega$ satisfy 
\begin{equation}\label{1}
\begin{cases} v_2\le v_1   & \mbox{in } \Omega \\ L_pv_2  \le f(x,v_2) & \mbox{in } \Omega\\  L_p v_1\ge f(x,v_1) & \mbox{in } \Omega. \end{cases}
\end{equation}
Set $
F=\{x\in \Omega: v_2(x)=v_1(x)\}$. If the structure conditions \eqref{structure condition} and hypothesis \eqref{f} are satisfied  and  $\emptyset \neq F \neq \Omega$,
then for every $y\in F^{*}$ and $\mathbf{v} \perp F$ at $y$, it follows that
\[
\langle X_i(y), \mathbf{v}\rangle=0
\]
for all $i=1,\cdots, m.$
\end{lem}

\begin{proof}
We argue by contradiction and 
suppose  that there exists $y\in F^{*}$, a vector $\mathbf{v}\perp F$ at $y$, and $i\in \{1, \cdots m\}$ such that 
$\sigma_i(y):=\langle X_i(y), \mathbf{v} \rangle \neq 0$.  We denote by $\sigma(x) $ the vector field $\sigma(x)= (\sigma_1(x),...,\sigma_m(x))$, and note that for $\mathbf{v}$ fixed, this is a smooth vector field on $\Omega$.

Let $z=y+\mathbf{v}$ and $r=|\mathbf{v}|$. We denote by $|x-z|$ the Euclidean distance between the points $x,z$ and proceed to define $\tilde{b}(x)=e^{-\alpha |x-z|^2}$, and
\[
b(x)=\alpha^{-2}(\tilde{b}(x)-e^{-\alpha r^2})
\]
in $\Omega$ where the value of the positive constant $\alpha$  is to be determined later.  Choose a neighborhood $V$ of $y$ such that $0<|\sigma(x)|$ for  $x\in \overline{V}\subset \Omega$ and  denote by $M_1,M_2, M_3, M_4$ positive constants depending on $v_2$ and $F$, such that for every $x\in \overline{V}$ one has $|X_j \sigma_i(x)|\le M_1$, 
$ |X_jX_i(b+v_2) (x)|\le M_2$, and $M_4\le |\sigma(x)|\le M_3$
for $i,j=1,\cdots, m$.

By a direct calculation, one can deduce
\[
X_i b(x)=-2\alpha^{-1} \tilde{b}(x)\sigma_i(x),
\]
\[
|X b(x)|=2\alpha^{-1}\tilde{b}(x)|\sigma(x)|=2\alpha^{-1}\tilde{b}(x)\big(\sum_{i=1}^m \sigma_i(x)^2\big)^{1/2},
\]
\[
X_jX_i b(x)=\tilde{b}(x)(4\sigma_j\sigma_i-2\alpha^{-1}X_j\sigma_i(x)).
\]

Substituting the identities above in the  expression for  $L_p b$ yields
\[
\begin{aligned}
L_p b(x)&= - \sum_{j=1}^m \sum_{i=1}^m \frac{\partial A_j}{\partial\xi_i}(X b) X_jX_i b + d_j A_j(Xb) \\
&= -\tilde{b}(x)\sum_{i,j=1}^m \Big(4\frac{\partial A_j}{\partial\xi_i}(X b)\sigma_j\sigma_i-2\alpha^{-1} \frac{\partial A_j}{\partial\xi_i}(X b) X_j \sigma_i \Big) + d_j A_j(Xb).
\end{aligned}
\]

Applying the structure conditions \eqref{structure condition} of $A_j$, it follows that for every $x\in \overline{V}$,

\[
\begin{aligned}
L_p b(x) &= -\tilde{b}(x)\sum_{i,j=1}^m \Big(4\frac{\partial A_j}{\partial\xi_i}(X b)\sigma_j\sigma_i-2\alpha^{-1} \frac{\partial A_j}{\partial\xi_i}(X b) X_j \sigma_i \Big) + d_j A_j(Xb)\\
&\le  - \tilde{b}(x)\Big(4\beta(\kappa+|X b|)^{p-2}|\sigma|^2-2\alpha^{-1} M_1\gamma(\kappa+|X b|)^{p-2} \Big) + C(\kappa+|Xb|)^{p-2}|Xb|\\
&= -\tilde{b}(x)(\kappa+|X b|)^{p-2}\bigg(4\beta|\sigma|^2-2\alpha^{-1} M_1\gamma -C \alpha^{-1}|\sigma(x)|)\bigg).
\end{aligned}
\]
Similarily, 
\[
\begin{aligned}
\sum_{i, j=1}^m  \frac{\partial A_j}{\partial\xi_i}(X v_2) X_jX_i b&\ge  \tilde{b}(x)(\kappa+|X v_2|)^{p-2}\bigg(4\beta|\sigma|^2-2\alpha^{-1}M_1\gamma \bigg)\\
& \ge  \tilde{b}(x)(\kappa+|X v_2|)^{p-2}\bigg(4\beta M_4^2-2\alpha^{-1}M_1\gamma \bigg).
\end{aligned}
\]
In view of the non-vanishing hypothesis on $|Xv_2|$, there exist $\alpha_1$ and a positive constant $\epsilon_1$ such that for $\alpha\ge \alpha_1$ and $x\in \overline{V}$
\[
|X b(x)|\le \frac{1}{2} |X v_2(x)|,
\]
\[
L_p b(x) \le 0,
\]
\begin{equation}
\sum_{i, j=1}^m  \frac{\partial A_j}{\partial\xi_i}(X v_2) X_jX_i b(x)\ge \epsilon_1\tilde{b}(x).
\end{equation}

Since $A_j(\xi)$ is smooth in $\mathbb{R}^n\setminus \{0\}$, there exists positive constants $C, \epsilon_2$ such that
\begin{equation}
\sum_{i,j=1}^m|\frac{\partial A_j}{\partial\xi_i}(X (b+v_2))-\frac{\partial A_j}{\partial\xi_i}(X v_2)|\le C|X b|\le \epsilon_2 \alpha^{-1} \tilde{b}(x)
\end{equation}
for $x\in \overline{V}$. Thus,

\[
\begin{aligned}
L_p (b+v_2)&=-\sum_{i, j=1}^m  \frac{\partial A_j}{\partial\xi_i}(X (b+v_2)) X_jX_i (b+v_2)+d_j A_j(Xb+Xv_2)\\
&=-\sum_{i, j=1}^m \Big(\frac{\partial A_j}{\partial\xi_i}(X (b+v_2))-\frac{\partial A_j}{\partial\xi_i}(X v_2)+\frac{\partial A_j}{\partial\xi_i}(X v_2)\Big)X_jX_i (b+v_2) +d_j A_j(Xb+Xv_2)\\
&=-\sum_{i, j=1}^m \Big(\frac{\partial A_j}{\partial\xi_i}(X (b+v_2))-\frac{\partial A_j}{\partial\xi_i}(X v_2)\Big)X_jX_i (b+v_2)\\
&-\sum_{i, j=1}^m\frac{\partial A_j}{\partial\xi_i}(X v_2)X_jX_i b-\sum_{i, j=1}^m\frac{\partial A_j}{\partial\xi_i}(X v_2)X_jX_i v_2  +d_j A_j(Xb+Xv_2)\\
&\le M_2\epsilon_2\alpha^{-1}\tilde{b}(x)-\epsilon_1\tilde{b}(x)+L_p v_2 -d_jA_j(Xv_2) +d_j A_j(Xb+Xv_2)  \\
&\le  ( -\epsilon_1+M_2\epsilon_2\alpha^{-1})\tilde{b}(x)  +f(v_2)  +|d_j || A_j(Xb+Xv_2)-A_j(Xv_2)|\\
&\le ( -\epsilon_1+M_2\epsilon_2\alpha^{-1} + C\alpha^{-1}|\sigma(x)| )\tilde{b}(x)  +f(x,v_2) \\
&\le ( -\epsilon_1+M_2\epsilon_2\alpha^{-1} + C\alpha^{-1}|\sigma(x)| )\tilde{b}(x)  +|f(x,b+v_2)-f(x,v_2)| +f(x,b+v_2)  \\
\text{(By $(ii)$ in \eqref{f})} &\le ( -\epsilon_1+M_2\epsilon_2\alpha^{-1} + C\alpha^{-1}|\sigma(x)| )\tilde{b}(x)  +L|b| +f(x,b+v_2)  \\
\end{aligned}
\]
We can  now choose $\alpha\ge \alpha_1$ such that $L_p (b+v_2)\le f(x,b+v_2)$ on $\overline{V}$.

Next, we let $U=V\cap B(z, r)$ and   express its boundary as the union of two components $$\partial U=\Gamma_1\cup \Gamma_2,$$ where $\Gamma_1=\overline{B(z,r)}\cap \partial V$ and $\Gamma_2=\overline{V}\cap \partial B(z,r)$.

For $x\in \Gamma_1\subset \Omega\setminus F$, we have $v_2(x)<v_1(x)$. Choose $\alpha$ be  sufficiently large so that $v_2(x)+ b(x)\le v_1(x)$ on $\Gamma_1$ and $L_p(v_2+b)\le f(x,b+v_2)$ on $U$. On the other hand, since $b(x)=0$ when $x\in \Gamma_2$,  then the estimate $v_2(x)+ b(x)\le v_1(x)$ also holds on $\Gamma_2$. Thus one eventually obtains

\begin{equation}
\begin{cases} v_2+ b \le v_1   & \mbox{in } \partial U \\ L_p(v_2+ b)  \le f(x,b+v_2) & \mbox{in } U\\  L_p v_1\ge f(x,v_1) & \mbox{in } U. \end{cases}
\end{equation}
The Weak Comparison Principle in Lemma \ref{weak} implies that $v_2+ b\le v_1$ in $U$. Since $y$ is a maximum point of $v_2-v_1$ in $\Omega$, then necessarily its gradient at $y$ must vanish, i.e. $\nabla (v_2-v_1)(y)=0$. Finally we invoke the $C^1$ regularity of $v_1$ near the contact set and we observe that

\[
\begin{aligned}
0=\langle \mathbf{v}, \nabla(v_2-v_1)(y)\rangle&= \lim_{t\to 0^{+}}\frac{v_2(y+t\mathbf{v})-v_1(y+t\mathbf{v})-(v_2(y)-v_1(y))}{t}\\
&\le - \langle \mathbf{v}, \nabla b(y)\rangle\\
&=-2\alpha^{-1}r^2e^{-\alpha r^2} < 0.
\end{aligned}
\]
Since we have arrived at a contradiction the proof is complete. 

\end{proof}

By a similar argument, a Hopf-type maximum principle can be established. 

\begin{lem}{(A Hopf-type Maximum Principle)}\label{hopf} Let $\Omega\subset \mathbb{R}^n$ be an open and connected set  and $v\in C^2(\Omega)$ satisfy 
\begin{equation}\label{1}
\begin{cases} v\ge 0   & \mbox{in } \Omega \\ L_pv  \ge f(x,v) & \mbox{in } \Omega. \end{cases}
\end{equation}
Set $
F=\{x\in \Omega: v(x)=0\}$. If the structure conditions \eqref{structure condition} and hypothesis \eqref{mf} are satisfied  and  $\emptyset \neq F \neq \Omega$,
then for every $y\in F^{*}$ and $\mathbf{v} \perp F$ at $y$, it follows that
\[
\langle X_i(y), \mathbf{v}\rangle=0
\]
for all $i=1,\cdots, m.$
\end{lem}

\begin{proof}
We argue by contradiction and 
suppose  that there exists $y\in F^{*}$, a vector $\mathbf{v}\perp F$ at $y$, and $i\in \{1, \cdots m\}$ such that 
$\sigma_i(y):=\langle X_i(y), \mathbf{v} \rangle \neq 0$.  We denote by $\sigma(x) $ the vector field $\sigma(x)= (\sigma_1(x),...,\sigma_m(x))$, and note that for $\mathbf{v}$ fixed, this is a smooth vector field on $\Omega$.

Let $z=y+\mathbf{v}$ and $r=|\mathbf{v}|$. We denote by $|x-z|$ the Euclidean distance between the points $x,z$ and proceed to define $\tilde{b}(x)=e^{-\alpha |x-z|^2}$, and
\[
b(x)=k(\tilde{b}(x)-e^{-\alpha r^2})
\]
in $\Omega$ where the value of the positive constant $k$ and $\alpha$ are to be determined later.  Choose a neighborhood $V$ of $y$ such that $0<|\sigma(x)|$ for  $x\in \overline{V}\subset \Omega$ and  denote by $M_1, M_2, M_3$ positive constants depending on $v_2$ and $F$, such that for every $x\in \overline{V}$ one has $|X_j \sigma_i(x)|\le M_1$ and $M_2\le|\sigma(x)|\le M_3$ for $i,j=1,\cdots, m$.

Elementary calculations and hypothesis \eqref{mf} show that for $\alpha$ sufficiently large, 
\[
\begin{aligned}
L_p b(x)&= - \sum_{j=1}^m \sum_{i=1}^m \frac{\partial A_j}{\partial\xi_i}(X b) X_jX_i b + d_j A_j(Xb) \\
&\le -k\tilde{b}(x) \alpha^2 (\kappa+|X b|)^{p-2}
\bigg(4\beta|\sigma|^2-2\alpha^{-1} |X\sigma| \gamma - 2\sup_{\overline V} |d|
|\sigma(x)| \alpha^{-1} \bigg)\\
&= -k\tilde{b}(x) \alpha^2 (\kappa+2\alpha |\sigma(x)|k\tilde b(x) )^{p-2}
\bigg[4\beta M_2^2-2\alpha^{-1} M_1\gamma -C M_3 \alpha^{-1}\bigg]
\\
& \text{ (choosing  }\alpha\text{ sufficiently large we may assume that the expression in brackets is larger than }M_2\beta)
\\
&\le - \alpha \beta   |b(x)| (\kappa+|b(x)|)^{p-2} 
\\
&\le - \bar C  |b(x)| (\kappa+|b(x)|)^{p-2} \le  f(x,b(x))  
\end{aligned}
\]
 for every $x\in \overline{V}$.
%
Next, we let $U=V\cap B(z, r)$ and   express its boundary as the union of two components $$\partial U=\Gamma_1\cup \Gamma_2,$$ where $\Gamma_1=\overline{B(z,r)}\cap \partial V$ and $\Gamma_2=\overline{V}\cap \partial B(z,r)$.

For $x\in \Gamma_1\subset \Omega\setminus F$, we have $v(x)>0$. Choose $k$ be  sufficiently 
small so that $b(x)\le v(x)$ on $\Gamma_1$. On the other hand, since $b(x)=0$ when $x\in \Gamma_2$,  then the estimate $b(x)\le v(x)$ also holds on $\Gamma_2$. Thus one eventually obtains

\begin{equation}
\begin{cases}  b(x)\le v(x)  & \mbox{in } \partial U \\ L_p( b)  \le f(x,b) & \mbox{in } U\\  L_p v\ge f(x,v) & \mbox{in } U. \end{cases}
\end{equation}
The Weak Comparison Principle in Lemma \ref{weak} implies that $b(x)\le v(x)$ in $U$. Since $y$ is a minimum point of $v(x)$ in $\Omega$, then necessarily its gradient at $y$ must vanish, i.e. $\nabla v(y)=0$. Finally we observe that in view of the $C^1$ regularity of $v$, one has 

\[
\begin{aligned}
0=\langle \mathbf{v}, \nabla v(y)\rangle&= \lim_{t\to 0^{+}}\frac{v(y+t\mathbf{v})-v(y)}{t}\\
&\ge  \lim_{t\to 0^{+}}\frac{b(y+t\mathbf{v})-b(y)}{t}\\
&=2k\alpha r^2e^{-\alpha r^2} > 0,
\end{aligned}
\]
arriving at a contradiction. 
 \end{proof}

In view of the Hopf-type comparison principle and of Theorem \ref{B2}, we deduce that the contact set $F=\{v_2=v_1\}$ must be either all of $\Omega$ or  the empty set, thus completing the proof of the strong comparison principle  in Theorem \ref{main}.

Likewise, the strong maximum principle theorem \ref{max} follows from the Hopf-type maximum principle.


\begin{thebibliography}{99}
%
%
\bibitem{BHT}
Z. Balogh, I. Holopainen, J. Tyson, 
\newblock Singular solutions, homogeneous norms, and quasiconformal mappings in Carnot groups. 
\newblock {\em Math. Ann.}
324 (2002), no. 1, 159-186.

\bibitem{Bony}
J.-M. Bony,
\newblock {\em Principe du maximum, in\'egalite de Harnack et unicit\'e du probl\'eme de Cauchy pour les operateurs elliptiques d\'eg\'en\'er\'es.} (French)  \newblock{\em Ann. Inst. Fourier (Grenoble)} 19 (1969), fasc. 1, 277-304 xii.


\bibitem{CDG1} L. Capogna, D. Danielli, N. Garofalo, \newblock{\em An embedding theorem and the Harnack inequality for nonlinear subelliptic equations}, \newblock{\em 
Comm. Partial Differential Equations},18 (1993), no. 9-10, 1765Ð1794. 

\bibitem{CCHY} J-H Cheng, H-L Chiu, J-F Hwang, P. Yang, \newblock{\em Strong maximum principle for mean curvature operators on subriemannian manifolds}, arXiv:1611.02384 (2016), 1--43.

\bibitem{HH} J. Heinonen, I. Holopainen, \newblock{\em Quasiregular maps on Carnot groups}, \newblock{ \em   J. Geom. Anal.}, (1997), 7--109. 
\bibitem{Hormander}
L. H\"ormander, \newblock{\em Hypoelliptic second order differential equations}, \newblock{ \em Acta Math,} 119 (1967), 147--171.

%
%

\bibitem{Tolksdorf}
P. Tolksdorf, 
\newblock {\em On the Dirichlet problem for quasilinear equations in domains with conical boundary points.} Comm. Partial Differential Equations 8 (1983), no. 7, 773-817. 


\bibitem{YN} Z. Yuan, P. Niu, \newblock{\em  A Hopf type principle and a strong maximum principle for the $p$-sub-Laplacian on a group of Heisenberg type}, \newblock{ \em Journal of mathematical research and exposition}, 27 (2007), no. 3, 605Ð-612.




%
%
%
%
%
%
%
%
%
\end{thebibliography}
\end{document}